\documentclass[11pt, a4paper, twoside]{amsart}

\usepackage{amsmath}
\usepackage{amssymb}
\usepackage{amsthm}

\usepackage{graphicx}
\usepackage{subcaption}

\usepackage{xcolor}

\usepackage{todonotes, changes}

\allowdisplaybreaks 

\newcommand{\R}{\mathbb{R}}
\newcommand{\Z}{\mathbb{Z}}
\newcommand{\N}{\mathbb{N}}

\newcommand{\K}{\mathcal{K}}

\newcommand{\conv}{\mathrm{conv}}
\newcommand{\spann}{\mathrm{span}}
\newcommand{\supp}{\mathrm{supp}}

\newcommand{\bd}{\mathrm{bd}}
\newcommand{\inter}{\mathrm{int}}

\newcommand{\vol}{\mathrm{vol}}
\newcommand{\suk}{\mathrm{h}}
\newcommand{\cs}{\mathrm{cs}}
\newcommand{\norm}[2]{\left|#1\right|_{\cs(#2)}}
\newcommand{\sukz}{\lambda}

\newcommand{\LE}{\mathrm{G}}

\newcommand{\blaschke}{\mathrm{sh}}

\usepackage{hyperref}
\hypersetup{backref, pdfpagemode=FullScreen, colorlinks=true,
  citecolor=magenta, linkcolor=cyan, urlcolor=blue}

\title[Lattice points, volume and successive minima]{Interpolating between volume and lattice point enumerator with successive minima}
\author{Ansgar Freyer and Eduardo Lucas}
\address{Technische Universität Berlin, Institut für Mathematik, Sekr. MA4-1, Stra{\ss}e des 17 Juni 136, D-10623 Berlin}
\address{Departamento de Matem\'aticas, Universidad de Murcia, Campus de Espinar\-do, 30100-Murcia, Spain}
\email{freyer@math.tu-berlin.de, eduardo.lucas@um.es}
\date{}
\thanks{The research of the second author is part of the project PGC2018-097046-B-I00, supported by MCIN/ AEI/10.13039/501100011033/ FEDER “Una manera de hacer Europa”.}

\numberwithin{equation}{section}

\usepackage{mathtools}
\mathtoolsset{showonlyrefs=true}

\begin{document}


\theoremstyle{plain}
\newtheorem{theorem}{Theorem}[section]
\newtheorem{lemma}[theorem]{Lemma}
\newtheorem{corollary}[theorem]{Corollary}
\newtheorem{conjecture}{Conjecture}
\newtheorem{proposition}[theorem]{Proposition}
\newtheorem*{question}{Question}
\newtheorem{thmx}{Theorem}
\renewcommand{\thethmx}{\Alph{thmx}}
\newtheorem{lemmax}[thmx]{Lemma}

\theoremstyle{definition}
\newtheorem*{definition}{Definition}

\newtheorem{remark}[theorem]{Remark}
\newtheorem{claim}{Claim}
\newtheorem*{remark*}{Remark}
\newtheorem{example}[theorem]{Example}

\begin{abstract} We study inequalities that simultaneously relate the number of lattice points, the volume and the successive minima of a convex body to one another. One main ingredient in order to establish these relations is Blaschke's shaking procedure, by which the problem can be reduced from arbitrary convex bodies to anti-blocking bodies. As a consequence of our results, we obtain an upper bound on the lattice point enumerator in terms of the successive minima, which is equivalent to Minkowski's upper bound on the volume in terms of the successive minima.
\end{abstract}

\subjclass{11P21, 52A10, 52A20, 52C05, 52C07}
\keywords{Lattice points in convex bodies, Successive minima, Minkowski's Second Theorem, Blaschke's shaking procedure}

\maketitle

\section{Introduction and Results}
\label{sec:intro}
Let $\K^n$ denote the class of all compact convex sets in $\R^n$ with non-empty interior.
For $K\in\K^n$ with $K=-K$, the $i$-th successive minimum is defined as 
\begin{equation}
\label{eq:lambdai}
\sukz_i(K) = \min\big\{ \sukz>0:\dim(\sukz K\cap\Z^n) \geq i\big\},
\end{equation}
for $i=1,\dots,n$. More generally, if $K$ is not necessarily symmetric, we define\linebreak $\sukz_i(K) = \sukz_i(\cs(K))$, where $\cs(K)=\frac{1}{2}(K-K)$. The successive minima have originally been introduced by Minkowski and he related them to the volume in the following way:
\begin{equation}
\label{eq:minkowski2ndtheorem}
\frac{1}{n!}\prod_{i=1}^n \frac{2}{\sukz_i(K)}\leq \vol(K)\leq\prod_{i=1}^n\frac{2}{\sukz_i(K)}.
\end{equation}
This classical result is known as Minkowski's Second Theorem on successive minima.
For origin-symmetric $K$, this has been proven by Minkowski \cite[Ch.2, Theorems 9.1 and 9.2]{geometryofnumbers}. For general $K\in\K^n$, the upper bound follows directly from the inequality $\vol(K) \leq \vol(\cs (K))$, which in turn is a special case of the Brunn-Minkowski inequality \cite[Ch.1, Theorem 1.7]{geometryofnumbers}. The lower bound can also be proved by an inclusion argument, similar to the symmetric case: One considers the convex hull of the $n$ segments in $K$ that realize the $\lambda_i(K)$ \cite[Remark 1.1]{henkhenzehernandez}.

Many alternatives to Minkowski's complicated original proof have been obtained. One of the first short proofs was given by Davenport \cite{davenportmink}. More analytic proofs were obtained by Weyl \cite{weyl} and Estermann \cite{estermann}; and Bambah, Woods and Zassenhaus provided three new proofs in \cite{bwz}. A more recent example was obtained by Henk \cite{Henk2002}.

The result has been extended, for instance, to more general successive minima by Hlawka \cite[Sec.~9.5]{geometryofnumbers}; to more general discrete sets, not necessarily lattices, by Woods \cite{woods}; to intrinsic volumes by Henk \cite{henk_intrinsic}; or to surface area measures by Henk, Henze and Hernández Cifre \cite{henkhenzehernandez}.

The lower and upper bound in \eqref{eq:minkowski2ndtheorem} are attained, e.g., by simplices and cubes respectively. Betke, Henk and Wills studied the relation of the lattice point enumerator $\LE(K) = |K\cap\Z^n|$ to the successive minima of $K$ and obtained for $K=-K$ that
\begin{equation}
\label{eq:minkowskibhw}
\frac{1}{n!}\prod_{i=1}^n\Big(\frac{1}{\sukz_i(K)}-1\Big) \leq \LE(K) \leq \prod_{i=1}^n\Big(\frac{2i}{\sukz_i(K)} +1\Big),
\end{equation}
where for the lower bound $\sukz_n(K)\leq 2$ is needed \cite[Proposition 2.1 and Corollary 2.1]{bhw}. While the lower bound is best-possible, it is conjectured that the upper bound can be strengthened as follows \cite[Conjecture 2.1]{bhw}:
\begin{conjecture}[Betke, Henk, Wills]
Let $K\in\K^n$ and $\sukz_i = \sukz_i(K)$. Then one has
\begin{equation}
\label{eq:discminkconj}
\LE(K)\leq \prod_{i=1}^n\Big\lfloor\frac{2}{\sukz_i}+1\Big\rfloor,
\end{equation}
where for a real number $x\in\R$, $\lfloor x \rfloor = \max\{z\in\Z: z\leq x\}$ denotes the floor function.
\end{conjecture}
Equality would be attained, e.g., for boxes of the form $[-m_1,m_1]\times...\times [-m_n,m_n]$, where $m_i\in\Z_{>0}$. In dimension 2 the conjecture has been confirmed by Betke, Henk and Wills themselves \cite[Theorem 2.2]{bhw} and in dimension 3 it has been shown by Malikiosis \cite[Section 3.2]{malikiosis}. Moreover Malikiosis also proved that  \cite[Theorem 3.2.1]{malikiosis} 
\begin{equation}
\label{eq:malikiosis}
\LE(K) \leq \frac{4}{e}\big(\sqrt{3}\big)^{n-1} \prod_{i=1}^n \Big(\frac{2}{\lambda_i}+1\Big),
\end{equation}
where again $\lambda_i=\lambda_i(K)$. To this day, \eqref{eq:malikiosis} is the best known upper bound for $\LE(K)$ in terms of the successive minima in general dimension.

Betke, Henk and Wills also pointed out in \cite[Proposition 2.2]{bhw} that any inequality of the form 
\begin{equation}
\label{eq:discminkconjweak}
\LE(K)\leq \prod_{i=1}^n\Big(\frac{2}{\sukz_i}+c_i\Big),
\end{equation}
for some numbers $c_i$, $1\leq i\leq n$, independent of $K$ (but not necessarily of $n$), would imply the upper bound in Minkowski's Second Theorem \eqref{eq:minkowski2ndtheorem}. Indeed, one can asymptotically approximate the volume of $K$ by the lattice point enumerator with respect to progressively finer lattices (using the properties of the Riemann integral), to which \eqref{eq:discminkconjweak} could then be applied, and the resulting limit is precisely Minkowski's bound. 

In this paper, we use Minkowski's Second Theorem to show \eqref{eq:discminkconjweak} with $c_i=n$ (cf.\ Corollary \ref{cor:discrminkn}). In order to do so, we aim to express the deviation between $\LE(K)$ and $\vol(K)$ in terms of the successive minima $\lambda_i(K)$, $i=1,...,n$.
Our approach stems from another conjecture by Betke, Henk and Wills that relates the volume, the lattice point enumerator and the successive minima simultaneously.

\begin{conjecture}[Betke, Henk, Wills]
\label{conj:bhw2}
Let $K\in\K^n$ and $\sukz_i = \sukz_i(K)$. Then,
\begin{equation}
\label{eq:conjupperbound}
\LE(K) \leq \vol(K) \prod_{i=1}^n \Big(1+\frac{i\,\sukz_i}{2}\Big)
\end{equation}
and, if $\lambda_n \leq \frac{2}{n}$,
\begin{equation}
\label{eq:conjlowerbound}
\LE(\inter K) \geq \vol(K) \prod_{i=1}^n \Big(1-\frac{i\,\sukz_i}{2}\Big),
\end{equation}
where $\inter K$ denotes the interior of $K$.
Moreover, if $K=-K$ and $\lambda_n\leq 2$, we have 
\begin{equation}
\label{eq:conjlowerboundsym}
\LE(\inter K) \geq \vol(K) \prod_{i=1}^n \Big(1-\frac{\sukz_i}{2}\Big). 
\end{equation}
\end{conjecture}
The bound \eqref{eq:conjlowerboundsym} is stated as Conjecture 2.2 in \cite{bhw}, where it is formulated for arbitrary $n$-dimensional lattices. However, there is no loss of generality in restricting to the integer lattice $\Z^n$. \eqref{eq:conjupperbound} and \eqref{eq:conjlowerbound} have been communicated to the authors by Martin Henk personally. In the general case, we obtain the following weakenings of \eqref{eq:conjupperbound} and \eqref{eq:conjlowerbound}:

\begin{theorem}
\label{thm:dimn}
Let $K\in\K^n$ and $\sukz_i=\sukz_i(K)$, $i\in [n]$. Then we have
\begin{equation}
\label{eq:upperboundn}
\LE(K) \leq \vol(K)\prod_{i=1}^n\left(1+\frac{n\sukz_i}{2}\right).
\end{equation}
Moreover, if $\sukz_n\leq \frac{2}{n}$, we have 
\begin{equation}
    \label{eq:lowerboundn}
    \LE(\inter K) \geq \vol(K) \prod_{i=1}^n\left(1-\frac{n\sukz_i}{2}\right).
\end{equation}
\end{theorem}

From this we can deduce immediately, by applying the upper bound in \eqref{eq:minkowski2ndtheorem} to the volume in \eqref{eq:upperboundn}, the following inequality:

\begin{corollary}
\label{cor:discrminkn}
Let $K\in\K^n$ and $\sukz_i=\sukz_i(K)$, $i\in [n]$. Then we have
\begin{equation}
\label{eq:discrmink}
\LE(K)\leq\prod_{i=1}^n\left(\frac{2}{\sukz_i} +n\right).
\end{equation}
\end{corollary}

While our bound is tight for convex bodies $rK$, $r\to\infty$, it is weaker than Malikiosis's bound \eqref{eq:malikiosis}, if, e.g., $\lambda_i(K) = 1/c$ for some fixed number $c>0$. Then our bound is of order $n^n$, while the bound in \eqref{eq:malikiosis} is of order $\sqrt{3}^n$.

The proof of Theorem \ref{thm:dimn} makes use of an inequality of Davenport \cite{davenport}, which states that for any convex body $K\in\K^n$ one has the bound 
\begin{equation}
\label{eq:mixvol}
\LE(K) \leq \sum_{k=1}^n\,\sum_{I\in {[n]\choose k}} \vol_{n-k}(K|L_I^\perp),
\end{equation}
where $L_I=\spann\{e_i:i\in I\}$ and $K|L_I^\perp$ denotes the orthogonal projection of $K$ on $L_I^\perp$. Schymura generalized Davenport's inequality and obtained for an arbitrary linearly independent set $\{b_1,...,b_n\}\subseteq\Z^n$ that 
\begin{equation}
\label{eq:schymura}
\LE(K) \leq \sum_{k=1}^n\,\sum_{I\in {[n]\choose k}} \vol_{n-k}(K|L_I^\perp)\vol_k(P_I),
\end{equation}
where $L_I=\spann\{b_i:i\in I\}$ and $P_I = \sum_{i\in I} [0,b_i]$ \cite[Lemma 1.1]{schymura}. We reverse \eqref{eq:schymura} in the following way.

\begin{theorem}
\label{thm:revdavi}
Let $K\in\K^n$ and let $b_1,...,b_n\in\Z^n$ be linearly independent. Then
\begin{equation}
\label{eq:revdavi}
\vol(K)\leq \sum_{I\subseteq [n]} \LE_{\Z^n|L_I^\perp}(\inter K | L_I^\perp)
\end{equation}
holds, where $L_I = \spann\{b_i:i\in I\}$ and $\LE_{\Z^n|L_I^\perp}$ denotes the lattice point enumerator with respect to the projected lattice $\Z^n|L_I^\perp$.  The inequality is tight.
\end{theorem}

The factor $\vol_k(P_I)$ in \eqref{eq:schymura} is hidden in the correspondingly higher density of $\Z^n|L_I^\perp$. In fact, one has $\det(\Z^n|L_I^\perp)\geq\vol_k(P_I)^{-1}$.

Apart from yielding discrete versions of Minkowski's Second Theorem, Conjecture \ref{conj:bhw2} is interesting in its own right; on the one hand, one can deduce the well-known formula 
\begin{equation}
\label{eq:vollimit}
\lim_{r\to\infty} \frac{\vol(rK)}{\LE(rK)} =1
\end{equation}
from it, since $\sukz_i(rK)$ tends to $0$ as $r\to\infty$. On the other hand, if $K$ contains an $n$-dimensional set of lattice points it follows that $\sukz_i(K)\leq 2$ holds, and, if $K=-K$, one has $\sukz_i(K)\leq 1$, $1\leq i \leq n$. Therefore, we retrieve the universal bounds
\begin{equation}
\label{eq:blichfeldt}
\LE(K) \leq (n+1)!\,\vol(K),
\end{equation}
for $K$ with $\dim(K\cap\Z^n) =n$, and
\begin{equation}
\label{eq:vandercorput}
\LE(\inter K) \geq 2^{-n}\,\vol(K),
\end{equation}
for $K=-K$ with $\dim(K\cap\Z^n) =n$ from Conjecture \ref{conj:bhw2}. These bounds essentially correspond to classical results of Blichfeldt \cite{blichfeldt} and van der Corput \cite[Ch.2, Theorem 7.1]{geometryofnumbers}. 

In fact, all inequalities in Conjecture \ref{conj:bhw2} have equality cases that are invariant with respect to integer scaling; \eqref{eq:conjupperbound} is tight, e.g., for integer multiples of the standard simplex $T_n=\conv\{0,e_1,...,e_n\}$, since $\sukz_i(T_n) =2$ and thus,
\begin{equation}
\label{eq:ehrhart}
\vol(mT_n)\prod_{i=1}^n\Big(1+\frac{i\sukz_i(mT_n)}{2}\Big) = \frac{1}{n!}\prod_{i=1}^n(m+i),
\end{equation}
where the right hand side is exactly the Ehrhart polynomial of $T_n$ \cite[Theorem 2.2 (a)]{beck}. In view of \cite[Theorem 2.2 (b)]{beck}, we have 
\begin{equation}
    \LE\big( \inter(mT_n)\big) = \frac{1}{n!}\prod_{i=1}^n (m-i) =\vol(mT_n)\prod_{i=1}^n\Big(1-\frac{i\sukz_i(mT_n)}{2}\Big)
\end{equation}
and so \eqref{eq:conjlowerbound} is tight for integer multiples of $T_n$ as well. As it has been mentioned already in \cite{bhw}, equality cases for \eqref{eq:conjlowerboundsym} are given for example by boxes parallel to the coordinate axes with integral side lengths.

In dimension 2, we can confirm the upper bound in Conjecture \ref{conj:bhw2}. For the non-symmetric lower bound we obtain an asymptotic confirmation:

\begin{theorem}
\label{thm:dim2}
Let $K\in\K^2$ and $\sukz_i=\sukz_i(K)$. Then we have 
\begin{equation}
    \label{eq:dim2upper}
    \LE(K) \leq \vol(K) \left(1+\frac{\sukz_1}{2}\right)\left(1+\sukz_2\right)
\end{equation}
and
\begin{equation}
    \label{eq:dim2lower}
    \LE(\inter K) \geq \vol(K)\left(1-\frac{\sukz_1}{2}-\sukz_2\right).
\end{equation}
\end{theorem}

In particular, for any $\varepsilon>0$, if $\sukz_1\leq\frac{2\varepsilon}{1+\varepsilon}$ it follows from \eqref{eq:dim2lower} that
\begin{equation}
    \label{eq:dim2lowereps}
 \LE(\inter K) \geq \vol(K)\left(1-\frac{\sukz_1}{2}\right)\left(1-(1+\varepsilon)\sukz_2\right).
\end{equation}




Moreover, \eqref{eq:conjlowerboundsym} can be confirmed for the special class of origin-symmetric lattice polygons.

\begin{proposition}
\label{prop:polygons}
Let $P\in\K^2$ be an origin-symmetric lattice polygon, i.e., we have $-P=P$ and $P$ is the convex hull of finitely many integer points. Then we have 
\begin{equation}
    \LE(\inter P) \geq \vol(P) \left(1-\frac{\sukz_1}{2}\right)\left(1-\frac{\sukz_2}{2}\right),
\end{equation}
where $\sukz_i=\sukz_i(P)$.
\end{proposition}

In dimension $n$, we also obtain the following bounds in terms of the covering radius $\mu(K)$ of $K\in\K^n$, i.e., the smallest number $\mu>0$ such that $\mu K+ \Z^n = \R^n$.
 
\begin{proposition}
\label{prop:covrad}
Let $K\in\K^n$ and $\mu = \mu(K)$. Then we have
\begin{equation}
\label{eq:covradup}
\LE(K) \leq\vol(K)\big(1+\mu\big)^n.
\end{equation}
If $\mu \leq 1$, i.e., $K+\Z^n =\R^n$, we also have
\begin{equation}
\label{eq:covraddown}
\LE(\inter K)\geq \vol(K)\big(1-\mu\big)^n.
\end{equation}
Both inequalities are tight.
\end{proposition}

The upper bound \eqref{eq:covradup} has also been shown independently by Dadush in \cite[Lemma 7.4.1]{dadush}.

The disadvantage of \eqref{eq:covradup} in comparison to the upper bound \eqref{eq:upperboundn} is that it cannot profit from $K$ being large in a lattice subspace. Consider the convex body $K=[-r,r]^{n-1}\times[-1/2,1/2]$, where $r$ is large. Then it holds that $\mu(K) = 1$, so the constant in \eqref{eq:covradup} is $2^n$. But the constant in \eqref{eq:upperboundn} is of order $n+1$, since $\lambda_i(K)$ tends to $0$ for $i<n$ as $r\to\infty$.

On the other hand, \eqref{eq:covraddown} is actually stronger than the lower bound \eqref{eq:lowerboundn} in Theorem \ref{thm:dimn}. We will use \eqref{eq:covraddown} to prove \eqref{eq:lowerboundn} in Section \ref{sec:dimn}.

Applied to the special class of convex lattice tiles, i.e., convex bodies $K$ with $K+\Z^n= \R^n$ and $\inter K \cap (z+\inter K) = \emptyset$, for all $z\in\Z^n\setminus\{0\}$, Proposition \ref{prop:covrad} yields for $r\geq 1$ that 
\begin{equation}
(r-1)^n \leq \LE(\inter (rK))\leq \LE(rK) \leq (r+1)^n,
\end{equation}
since $\vol(K)=\mu(K)=1$ and $\mu(rK) =\frac{\mu(K)}{r}$, which is sharp for $K=[0,1]^n$ and $r\in\Z_{>0}$.

The paper is organized as follows: In Section \ref{sec:pre} we introduce basic terms and facts from convex geometry and the geometry of numbers that are necessary for the proofs. In Section \ref{sec:anti-blocking} we consider the so-called anti-blocking convex bodies. These are convex bodies $K\subseteq\R^n_{\geq 0}$ such that for every $x=(x_1,\dots,x_n)\in K$ one has $\{x'\in\R^n_{\geq 0}:x_i'\leq x_i,\ \forall\;1\leq i\leq n\}\subseteq K$.  We will show that it is enough to prove Theorems \ref{thm:dimn} and \ref{thm:dim2} for anti-blocking bodies. Section \ref{sec:dimn} contains the proofs of Theorems \ref{thm:dimn} and \ref{thm:revdavi}, as well as Proposition \ref{prop:covrad}. In Section \ref{sec:dim2} we prove Theorem \ref{thm:dim2} and Proposition \ref{prop:polygons}.

\section{Preliminaries}
\label{sec:pre}

For a vector $u\in\R^n\setminus\{0\}$ we denote by $u^\perp$ the hyperplane orthogonal to $u$ passing through the origin and by $|u|$ its Euclidean length. For two vectors $u,v\in\R^n$, we let $\langle u,v\rangle$ be the standard scalar product. For a linear subspace $L\subseteq\R^n$ and a set $A\subseteq\R^n$ we denote by $A|L$ the image of the orthogonal projection of $A$ onto $L$. We write $\{x\}|L =x|L$. The Minkowski sum of two non-empty sets $A,B\subseteq\R^n$ is denoted by $A+B = \{a+b:a\in A, b\in B\}$. Also, for a number $\lambda\in\R$, we write $\lambda A = \{\lambda a: a\in A\}$ and $-A = (-1)A$. The convex hull of a set $A$ is denoted by $\conv A$. If $A=\{x,y\}$, we write $[x,y]=\conv A$. Moreover, for $n\in\N$, we write $[n]=\{1,...,n\}$ and ${[n]\choose k}=\{I\subseteq [n]: |I|=k\}$. 

The class of all $n$-dimensional convex bodies, i.e., compact and convex sets, is denoted by $\K^n$. For $K\in\K^n$, we denote its support function by $\suk(K,\cdot)$, which is defined for $x\in\R^n$ as follows:
\begin{equation}
\suk(K,x) = \max\{\langle x, y\rangle: y\in K\}.
\end{equation}
If $K$ satisfies $-K=K$, its gauge function $|\cdot|_K$ is defined as 
\begin{equation}
|x|_K = \min\{r>0:x\in rK\},
\end{equation} 
where $x\in \R^n$. As an alternative to \eqref{eq:lambdai}, one can use the gauge function of $\cs(K)$ to define the successive minima. Namely, one has $\sukz_1(K) = \min \norm{z}{K}$, where $z$ ranges over $\Z^n\setminus\{0\}$, and $\sukz_i(K) = \min\norm{z}{K}$, where $z$ ranges over $\Z^n\setminus\spann\big(\sukz_{i-1}(K)\cs(K)\cap\Z^n\big)$.

A lattice $\Lambda\subseteq\R^n$ is the integral span of linearly independent vectors $x_1,...,x_k\in\R^n$, i.e., 
\begin{equation}
\Lambda = \Big\{\sum_{i=1}^k m_ix_i: m_1,...,m_k\in\Z\Big\}.
\end{equation}
We will use that for any $X\subseteq\Lambda$, the set $\Lambda|\spann(X)^\perp$ is also a lattice. For a set $A\subseteq\R^n$ and a lattice $\Lambda\subseteq\R^n$, we consider the lattice point enumerator $\LE_\Lambda(A) = |A\cap\Lambda|$. In most cases, we will be concerned with $\Lambda = \Z^n$ and we write $\LE(A) = \LE_{\Z^n}(A)$. We will also need the following result which is due to van der Corput \cite[Ch.2, Theorem 6.1]{geometryofnumbers}.
\begin{theorem}[van der Corput]\label{t:vdc}
Let $M\subseteq\R^n$ be a Jordan-measurable set. Then there exists a vector $t\in\R^n$ such that
\begin{equation}
\label{eq:vdc}
\vol(M) \leq \LE(M+t).
\end{equation}
\end{theorem}

\section{Anti-blocking convex bodies}
\label{sec:anti-blocking}

A convex body $K\subseteq\R^n_{\geq 0}$ is anti-blocking if for every $x=(x_1,\dots,x_n)\in K$ the set $\{x'\in\R^n_{\geq 0}:x_i'\leq x_i,\,\forall i\in [n]\}$ is also contained in $K$. Given the convexity of $K$, the latter condition is equivalent to $K\cap e_i^\perp = K|e_i^\perp$, for all $i\in [n]$. 

Anti-blocking bodies have been introduced in \cite{fulkerson}. Their volumes have been extensively studied in \cite{artstein}. In the discrete setting, the set of lattice points $K\cap\Z^n$ inside of an anti-blocking body $K$ is called a compressed set. Compressed sets have been considered in \cite{greentao} (in the context of sum-set estimates) and in \cite{veomett} (in the context of discrete isoperimetric inequalities).

The goal of this section is to prove the following statement:

\begin{theorem}
\label{prop:anti-blocking}
For any convex body $K\in\K^n$, there exists an anti-blocking convex body $A\subseteq\R^n_{\geq 0}$ such that the following holds:
\begin{enumerate}
\item $\vol(K) = \vol(A)$,
\item $\LE(K) \leq \LE(A)$,
\item $\LE(\inter K) \geq \LE(\inter A)$ and
\item $\sukz_i(K) \geq \sukz_i(A)$, for all $i\in [n]$.
\end{enumerate}
\end{theorem}

This shows that it is enough to prove \eqref{eq:conjupperbound} and \eqref{eq:conjlowerbound} for the special class of anti-blocking bodies. 

An important tool for the proof of Theorem \ref{prop:anti-blocking} is the Blaschke shaking of a convex body $K\in\K^n$ with respect to an oriented hyperplane $u^\perp$, $u\neq 0$, which is defined as 
\begin{equation}
\blaschke_u(K) = \bigcup_{x\in K|u^\perp} \left[ x,\, x+\frac{f_{u,K}(x)}{|u|}\cdot u\right],
\end{equation}
where $f_{u,K}(x)$ denotes the length of the preimage of $x$ under the orthogonal projection $K\to u^\perp$ (cf.\ Figure \ref{fig:blaschke}). The Blaschke shaking has been introduced in \cite{blaschke}. This process, which bares resemblance to Steiner's symmetrization, belongs to a wider class of transformations known as ``shakings''. These processes have been explored, for instance, to obtain discrete isoperimetric inequalities by Kleitman \cite{kleitman}, and more recently by Bollobás and Leader \cite{bollobasleader}. Stability results, akin to that of Gross for Steiner's symmetrization, have been obtained by Biehl \cite{biehl}, Schöpf \cite{schopf}, and more recently, Campi, Colesanti and Gronchi \cite{ccg}, for example. Other applications were obtained in \cite{uhrin} and \cite{ccg2}.

\begin{figure}[htb]
\centering
\includegraphics[width = .75\textwidth]{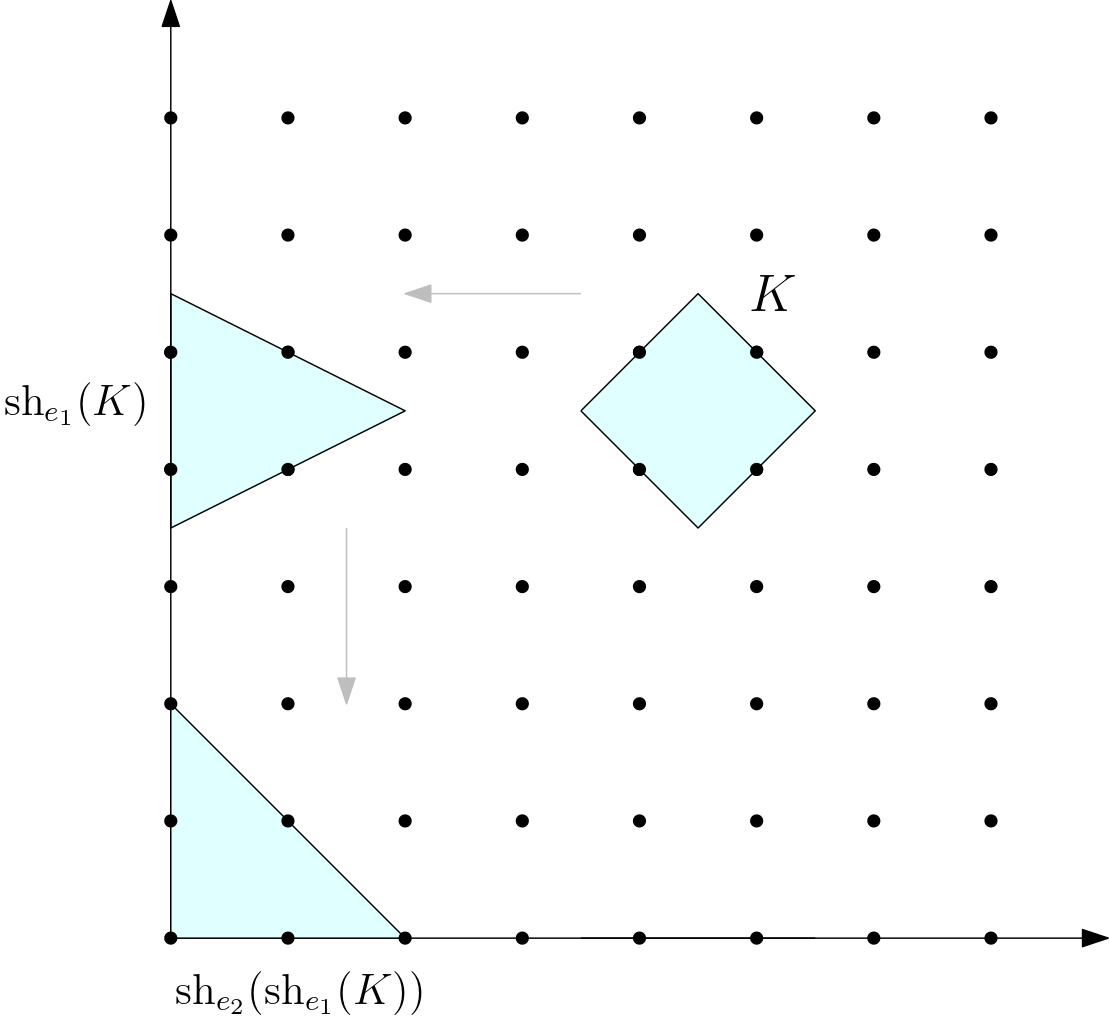}
\caption{Illustration of two consecutive Blaschke shakings. The body $\blaschke_{e_2}(\blaschke_{e_1}(K))$ is anti-blocking.}
\label{fig:blaschke}
\end{figure}

The operator $\blaschke_u$ is known to preserve convexity \cite[Lemma 1.1]{ccg} and we have the following lemma:

\begin{lemma}
\label{lemma:blaschke}
Let $K\in\K^n$ and $u\in\R^n\setminus\{0\}$. For the Blaschke shaking $\blaschke_u(K)$, the following relations hold:
\begin{enumerate}
\item $K|u^\perp\subseteq \blaschke_u(K)$
\item $\vol(K)=\vol(\blaschke_u(K))$,
\item $\norm{u}{K} = \norm{u}{\blaschke_u(K)}$,
\item $\norm{x}{K} \geq \norm{x|u^\perp}{\blaschke_u(K)}$, for all $x\in\R^n$.
\end{enumerate}
If $u=e_i$, for some $i\in [n]$, we also have
\begin{enumerate}
\setcounter{enumi}{4}
\item $\LE(K) \leq \LE(\blaschke_{e_i}(K))$,
\item $\LE(\inter K) \geq \LE(\inter(\blaschke_{e_i}(K))$.
\end{enumerate}

\end{lemma}

\begin{proof}
i) and ii) follow directly from the definition of $\blaschke_u(K)$. For iii), if $u=(r/2)(x-y)$ for any $x,y\in\R^n$ and any $r\in\R\setminus\{0\}$, projecting onto $u^\bot$ yields $x|u^\bot=y|u^\bot$. If we denote $\ell_z=z+\R u$ for $z\in K|u^\bot$, and $d_z=d(u^\bot, \ell_z\cap K)$ the signed Euclidean distance, then $d_x=d_y$ and by definition $z\in K$ if and only if $z-d_zu\in\blaschke_u(K)$. Therefore, given $r>0$ and considering $$ r\cdot\frac{x-y}{2} = r\cdot\frac{(x-d_xu)-(y-d_xu)}{2} $$ we obtain that $u\in r\,\cs(K)$ if and only if $u\in r\,\cs(\blaschke_u(K))$, i.e. iii).

For iv) let $r=\norm{x}{K}^{-1}$. Then there are $a,b\in K$ such that $rx = \frac{1}{2}(a-b)$ and from i) it follows that $$ r\cdot x|u^\perp = \frac{1}{2}\big(a|u^\perp - b|u^\perp\big) \in \cs\big(\blaschke_u(K)\big).$$ Thus $\norm{r\cdot x|u^\perp}{\blaschke_u(K)} \leq 1$, which implies iv) by the choice of $r$.

In order to prove v), we start with an interval $I=[a,b]\subseteq\R$ and show that the number of lattice points in an interval of length $b-a$ is maximized when $a\in\Z$. Otherwise, we could let $\delta = a-\lfloor a\rfloor$ and observe that
\begin{equation*}
    \begin{split}
        |(I-\delta)\cap\Z|&=\lfloor b-\delta\rfloor-\lceil a-\delta\rceil+1=\lfloor b-\delta\rfloor-\lfloor a\rfloor+1\\
        &\geq\lfloor b\rfloor-\lfloor a\rfloor = \lfloor b\rfloor - \lceil a\rceil +1=|I\cap\Z|
    \end{split}
\end{equation*}
In order to obtain v) it is then enough to note that the lattice points in $K$ and $\blaschke_{e_i}(K)$ are contained in intervals of the same lengths, while in $\blaschke_{e_i}(K)$, these intervals start at a lattice point and therefore contain at least as many lattice points as those in $K$ (cf.\ Figure \ref{fig:blaschke}).

vi) is proved with the same argument, but since the intervals involved are open, translating them such that they start at a lattice point will potentially reduce, but never increase, their lattice point count.
\end{proof}

\begin{proof}[Proof of Theorem \ref{prop:anti-blocking}]
Let $v_1,...,v_n\in\Z^n$ be linearly independent such that $\norm{v_i}{K} = \sukz_i(K)$. Since all the functionals involved are invariant with respect to unimodular transformations, we may assume that the matrix $[v_1,...,v_n]$ is an upper triangular matrix (e.g., a Hermite-normal-form \cite[Section 4.1]{schrijver}). Let $K_0=K$ and for $j\in [n]$, let $K_j=\blaschke_{e_j}(K_{j-1})$. We show that $A:=K_n$ is the desired body. To this end, we prove the following statement inductively.

\begin{claim}
\label{claim:1}
For $j\in\{0,...,n\}$, there exist linearly independent vectors $u_1,...,u_n\in\Z^n$ such that $\norm{u_i}{K_j} \leq \lambda_i(K)$ and the matrix $[u_1,...,u_n]$ is of the form
\begin{equation}
\label{eq:matrix}
\begin{pmatrix}
D_j & 0 \\
0 & T_{n-j}
\end{pmatrix},
\end{equation}
where $D_j$ is a $j\times j$-diagonal matrix and $T_{n-j}$ is an $(n-j)\times(n-j)$-upper triangular matrix. 
\end{claim}

For $j=0$, Claim \ref{claim:1} is clearly true with $u_i=v_i$, $1\leq i \leq n$. So we assume that the claim holds for some $j<n$. We choose $u_i^\prime = u_i|e_{j+1}^\perp$, for $i\neq j+1$, and $u_{j+1}^\prime = u_{j+1}$.  In view of $K_{j+1} =\blaschke_{e_{j+1}}(K_j)$, Lemma \ref{lemma:blaschke} iv) and our induction hypothesis, we have $$\norm{u_i^\prime}{K_{j+1}} = \norm{u_i|e_{j+1}^\perp}{K_{j+1}}\leq \norm{u_i}{K_j}\leq \lambda_i(K),$$ for all $i\neq j+1$. From Lemma \ref{lemma:blaschke} iii) it also follows that $$\norm{u_{j+1}^\prime}{K_{j+1}} = \norm{u_{j+1}}{K_j}\leq\lambda_{j+1}(K).$$ The matrix $[u_1^\prime,...,u_n^\prime]$ differs from $[u_1,...,u_n]$ only by the zeros in the $(j+1)$-th row after the diagonal entry. Therefore, the system $u_1^\prime,...,u_n^\prime\in\Z^n$ is also linearly independent and it fulfills the requirements of Claim \ref{claim:1} for $j+1$.

Hence, $A=K_n$ satisfies $\lambda_i(A) \leq \lambda_i(K)$ for all $i\in [n]$. The other requirements i)--iii) on $A$ follow from a repeated application of Lemma \ref{lemma:blaschke} ii),v) and vi). It remains to prove that $A$ is indeed anti-blocking. To this end, we use induction again to prove:

\begin{claim}
\label{claim:2}
 For $j\in\{0,...,n\}$ and $x\in K_j$, we have $x|e_i^\perp \in K_j$, for all $1\leq i\leq j$.
 \end{claim}

\begin{figure}[htb]
\centering
\includegraphics[width = .75\textwidth]{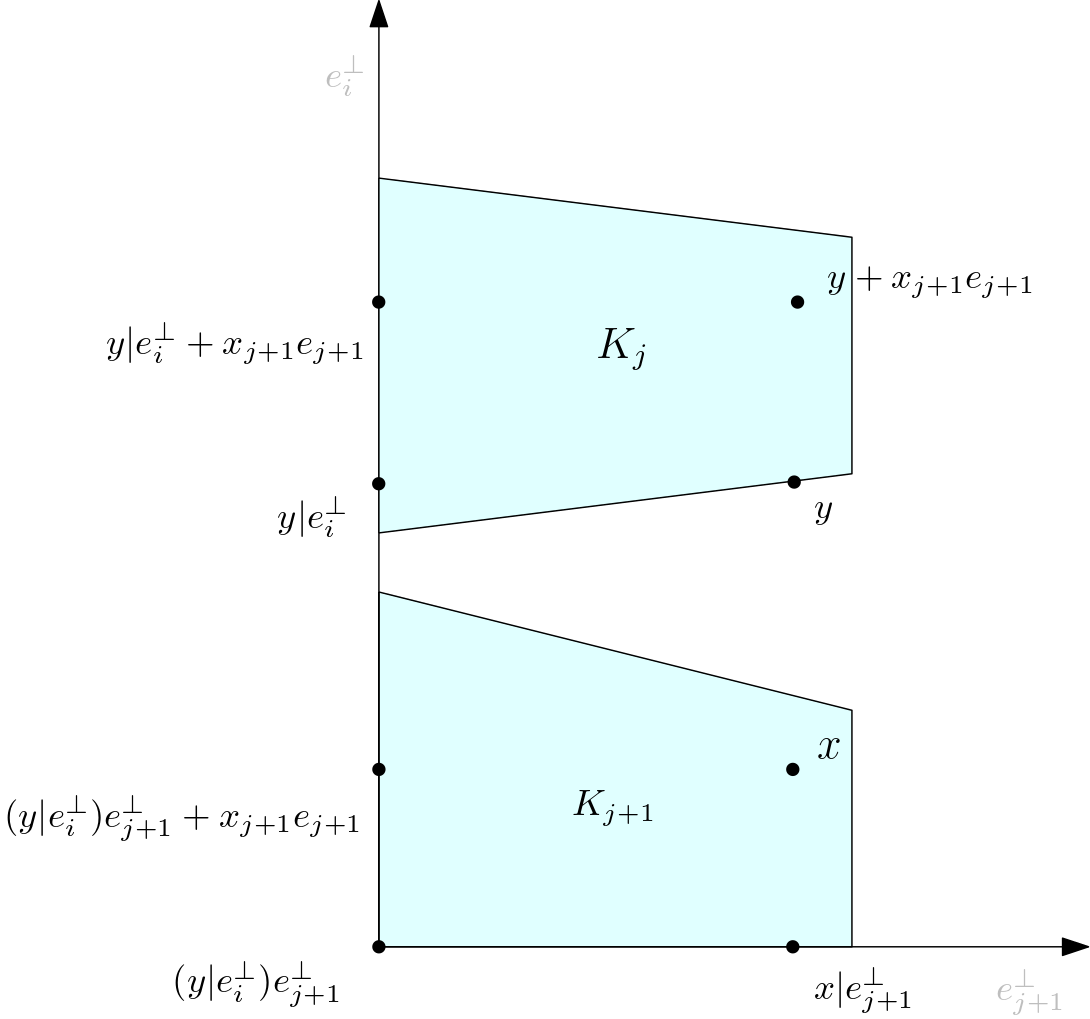}
\caption{The construction for the proof of Claim 2.}
\label{fig:anti-blocking}
\end{figure}

For $j=0$, the statement is trivial. So we assume Claim \ref{claim:2} holds for some $j\in [n]$. Let $x\in K_{j+1}$. By Lemma \ref{lemma:blaschke} i) it follows that $x|e_{j+1}^\perp\in K_{j+1}$. So we consider $i\in [j]$. Let $x_{j+1}$ be the $(j+1)$-th entry of $x$ and let $y\in K_j$ be the lowest (with respect to $e_{j+1}$) point in the preimage of $x|e_{j+1}^\perp$ under the orthogonal projection $K_j\rightarrow e_{j+1}^\perp$ (cf.\ Figure \ref{fig:anti-blocking}). Then, $[y, y+x_{j+1}e_{j+1}]\subseteq K_j$. By induction, it follows that $[y|e_i^\perp, y|e_i^\perp+x_{j+1}e_{j+1}]\subseteq K_j$. Since $(y|e_i^\perp)|e_{j+1}^\perp = (x|e_i^\perp)|e_{j+1}^\perp$, the interval $[(x|e_i^\perp)|e_{j+1}^\perp, (x|e_i^\perp)|e_{j+1}^\perp+x_{j+1}e_{j+1}]$ is contained in $K_{j+1}$. Since $(x|e_i^\perp)|e_{j+1}^\perp+x_{j+1}e_{j+1} = x|e_i^\perp$, Claim \ref{claim:2} holds for $j+1$.

For $j=n$, Claim \ref{claim:2} yields that $A$ is anti-blocking.

\end{proof}

The proof of Claim \ref{claim:2} essentially corresponds to the argument given in the proof of Lemma 1.2 in \cite{ccg}.

One of the reasons why anti-blocking bodies are beneficial when dealing with successive minima problems is that the successive minima are always realized by the standard basis of $\Z^n$. 

\begin{lemma}
\label{lemma:succminantiblocking}
Let $K\in\K^n$ be anti-blocking. Then the coordinates can be permuted in such a way that $\norm{e_i}{K}=\sukz_i(K)$ holds. In this case, one also has $\frac{2}{\sukz_i(K)}e_i\in K$, $1\leq i \leq n$.
\end{lemma}

\begin{proof}
Let $v_1,...,v_n\in\Z^n$ be linearly independent with $\norm{v_i}{K}=\sukz_i(K)$. Then there exists a permutation $\sigma$ of $[n]$ such that the $\sigma_i$-th entry of $v_i$ is non-zero. Otherwise the determinant of $[v_1,...,v_n]$ would be zero, a contradiction. For the sake of simplicity we assume that $\sigma$ is the identity. Since $K$ is anti-blocking, the projection $w_i$ of $v_i$ on $\spann\{e_i\} = \bigcap_{j\neq i}e_j^\perp$ is contained in $K$ and a repeated application of Lemma \ref{lemma:blaschke} iv) shows that $\norm{w_i}{K}\leq\norm{v_i}{K}=\sukz_i(K)$. By the minimality of the $\sukz_i$'s and the fact that $w_i\in\spann\{e_i\}\cap\Z^n$, we obtain $w_i=e_i$ and $\norm{e_i}{K}=\sukz_i(K)$.

For the second part we deduce from $\norm{e_i}{K}=\lambda_i(K)$ that $$\frac{1}{\lambda_i(K)}e_i = \frac{1}{2}(a-b),$$ for some $a,b\in K$. Since $1/\sukz_i(K)$ is the maximal number $r$ such $re_i\in\cs(K)$, $b_i$ must be zero. So $b$ is a member of $e_i^\perp$  and since $K$ is anti-blocking we obtain $$\frac{2}{\sukz_i(K)}e_i = \Big(\frac{2}{\sukz_i(K)}e_i +b\Big)\Big|\spann\{e_i\} = a|\spann\{e_i\} \in K$$ as desired.
\end{proof}

\section{$n$-dimensional case}
\label{sec:dimn}
We start by proving Proposition \ref{prop:covrad}.

\begin{proof}[Proof of Proposition \ref{prop:covrad}]
For the upper bound, it is enough to show that $\mu K$ contains a measurable set $S$ with $S + \Z^n = \R^n$ and $\inter S \cap (z+\inter S)=\emptyset$, for all $z\in\Z^n\setminus\{0\}$. In that case, we have $$\LE(K)=\vol((K\cap\Z^n) + S) \leq \vol(K+\mu K) = (1+\mu)^n\vol(K).$$ In order to find $S$, let $P=[0,1]^n$. There are finitely many translates $\mu K+x_i$, $x_i\in\Z^n$, $1\leq i \leq m$, that cover $P$. We define inductively $P_1 = P \cap (\mu K+x_1)$ and $$P_i = \big( P\setminus (\bigcup_{j<i} P_j)\big) \cap (\mu K +x_i).$$ Now, let $S_i = P_i-x_i\subseteq \mu K$ and $S=\bigcup_{i=1}^m S_i$. We claim that $S$ is the desired set. To prove this, we show that $S$ has volume $1$ and that its $\Z^n$-translates do not overlap. 

Clearly the $P_i$'s are interiorly disjoint, i.e., $\inter (P_i) \cap \inter (P_j) = \emptyset$, and satisfy $\bigcup_{i=1}^m P_i =P$. The $S_i$'s are interiorly disjoint too; suppose there are $i\neq j$ such that $\inter S_i$ intersects $\inter S_j$. Then, $\inter P_i$ intersects $\inter P_j +x_i-x_j$. Since the $\Z^n$ translates of $P$ are interiorly disjoint, we must have $x_i=x_j$, a contradiction. Therefore the $S_i$'s are interiorly disjoint and it follows that
\begin{equation}
\vol(S)=\sum_{i=1}^m\vol(S_i)=\sum_{i=1}^m\vol(P_i)=\vol(P)=1.
\end{equation}
Now assume that $\inter S$ intersects $\inter S+ x$ for some $x\in\Z^n$. Then there exist $i,j\in [m]$ such that $\inter P_i-x_i$ intersects $\inter P_j-x_j+x$. Again, since the $\Z^n$-translates of $P$ are interiorly disjoint, as well as the $P_i$'s, we must have $i=j$ and $x=0$. Hence, the $\Z^n$-translates of $S$ are interiorly disjoint and so $S$ is as desired. This finishes the proof of the upper bound.

For the lower bound, we apply \eqref{eq:vdc} to $K^\prime=(1-\mu) \inter K$ and obtain a vector $t\in\R^n$ such that $\vol(K^\prime)\leq \LE(K^\prime+t)$. Since $\mu K+\Z^n = \R^n$, we may assume that $t\in\mu K$ holds. Thus, since $\mu\leq 1$, 
\begin{equation}
\vol(K)(1-\mu)^n = \vol(K^\prime)\leq\LE(K^\prime+t)\leq\LE\big((1-\mu)\inter K + \mu K) = \LE(\inter K).
\end{equation}
In order to see that the inequality is tight, consider $K = [0,m]^n$, where $m\in\Z_{>0}$. For  such cubes one has $\vol(K) = m^n$, $\LE(K)=(m+1)^n$, $\LE(\inter K) = (m-1)^n$ and $\mu(K) = 1/m$. So equality is achieved for both of the bounds.
\end{proof}

The strategy of finding an appropriate tiling used in the proof of the upper bound above has also been applied, for instance, in the proof of Blichfeldt's classical variant of Theorem~\ref{t:vdc} \cite[Ch.2, Theorem 5.2]{geometryofnumbers}. Moreover, in \cite{xue}, the authors showed that convex tilings in these conditions need not exist. 

Next, we come to the proof of Theorem \ref{thm:dimn}.

\begin{proof}[Proof of Theorem \ref{thm:dimn}]
In order to prove \eqref{eq:upperboundn}, we may assume that $K$ is anti-blocking by Theorem \ref{prop:anti-blocking}. After renumbering the coordinates, we can also assume that $\norm{e_i}{K}=\sukz_i$, $1\leq i \leq n$ holds (cf.\ Lemma \ref{lemma:succminantiblocking}).  For a set $I\subseteq[n]$, let $L_I = \spann\{e_i:i\in I\}$.  
An inequality of Rogers and Shephard \cite[Theorem 1]{rogers-shephard} yields that 
\begin{equation}
\label{eq:rogers-shephard}
\vol_k(K\cap L_I)\vol_{n-k}(K | L_I^\perp) \leq { n\choose k}\vol(K),
\end{equation}
for any $I\in {[n]\choose k}$. By Lemma \ref{lemma:succminantiblocking} we have $\frac{2}{\sukz_i} e_i \in K$, so from \eqref{eq:rogers-shephard}, we deduce that
\begin{equation}
\begin{split}
\vol_{n-k}(K|L_I^\perp) & \leq {n\choose k} \,\frac{\vol(K)}{\vol_k(K\cap L_I)} \leq {n\choose k}\, \frac{\vol(K)}{\vol_k(\conv\{(2/\sukz_i)e_i:i\in I\})}\\
&=k! {n\choose k} \vol(K) \prod_{i\in I} \frac{\sukz_i}{2}\leq \vol(K) \prod_{i\in I} \frac{n\sukz_i}{2}.
\end{split}
\end{equation}
Combining this with Davenport's inequality \eqref{eq:mixvol} yields
\begin{equation}
\LE (K) \leq \vol(K) \sum_{I\subseteq [n]}  \prod_{i\in I} \frac{n\sukz_i}{2} = \vol(K)\prod_{i=1}^n\left(1+\frac{n\sukz_i}{2}\right)
\end{equation}
as desired. In order to prove \eqref{eq:lowerboundn}, we use the lower bound \eqref{eq:covraddown} in terms of $\mu(K)$, as well as the relation \cite[Lemma 2.4]{coveringminima} 
\begin{equation}
\label{eq:lambdamu}
    \mu(K) \leq \sum_{i=1}^n \frac{\sukz_i}{2}.
\end{equation}
Since $\sukz_n \leq 2/n$, \eqref{eq:lambdamu} yields that $\mu(K)\leq 1.$ Thus, we may apply Proposition \ref{prop:covrad} and obtain
\begin{equation}
    \begin{split}
        \LE(\inter K) &\geq \vol(K)\big(1-\mu(K))^n\geq \vol(K)\Big(1-\sum_{i=1}^n \frac{\sukz_i}{2}\Big)^n\\
        & = \vol(K)\Big(\frac{1}{n}\sum_{i=1}^n\Big(1-\frac{n\sukz_i}{2}\Big)\Big)^n\\
        & \geq \vol(K) \prod_{i=1}^n \Big(1-\frac{n\sukz_i}{2}\Big),
    \end{split}
\end{equation}
where we used the inequality of arithmetic and geometric means in the last step.
\end{proof}

Since $K$ is anti-blocking in our case, \eqref{eq:mixvol} can also be derived directly as follows:

\begin{equation}
\begin{split}
\LE(K) & = \vol\big((K\cap\Z^n) + [-1,0]^n\big)\leq \vol(K+[-1,0]^n)\\
& =\sum_{k=1}^n\,\sum_{I\in {[n]\choose k}} \vol_{n-k}(K|L_I^\perp),
\end{split}
\end{equation}
where the last equation follows, since $K$ is anti-blocking and thus, the Minkowski sum can be decomposed into a union of disjoint prisms: $$K+[-1,0]^n = \bigcup_{I\subseteq [n]} \left\{ x\in\R^n: x|L_I^\perp\in K|L_I^\perp\text{ and }x_i\in [-1,0],\,\forall i\in I \right\}.$$

Next, we come to the proof of our reverse version of Schymura's inequality \eqref{eq:schymura}.

\begin{proof}[Proof of Theorem \ref{thm:revdavi}]
For an ordered linearly independent set $B = \{b_1,...,b_k\}$ and  a vector $x=\sum_{i=1}^k\alpha_i b_i\in\spann\,B$, we write $\mathrm{supp}_B(x) = \{i:\alpha_i\neq 0\}$. We will show the following statement:

\begin{claim}
\label{claim:3}
Let $\Lambda$ be an $n$-dimensional lattice and $B=\{b_1,...,b_n\}\subseteq\Lambda$ be a linearly independent set. For any convex and bounded (but not necessarily closed) set $K\subseteq\mathrm{span}(\Lambda)$, and any $t\in\mathrm{span}(\Lambda)$, it holds that
\begin{equation}\label{eq:davinduction}
    \LE_\Lambda(K+t) \leq \sum_{I\subseteq \mathrm{supp}_B(t)} \LE_{\Lambda | L_I^\perp}(K|L_I^\perp),
\end{equation} 
where $L_I = \mathrm{span} \{b_i : i\in I\}.$
\end{claim}

If $t=0$ there is only one summand in \eqref{eq:davinduction} corresponding to $I=\emptyset$, and so \eqref{eq:davinduction} reads as $\LE_\Lambda(K)\leq\LE_\Lambda(K)$, a tautology. Thus, from now on we assume that $t\neq0$.

First we note that if $n=1$, then \eqref{eq:davinduction} states for non-zero $t$ that \begin{equation}\label{eq:davinduction1.5}\LE_\Lambda(K+t) \leq \LE_\Lambda(K) + 1.\end{equation} Since any convex body $K\subseteq\R^1$ is an interval, the statement is confirmed. 

Now, for any $n>1$, we will prove \eqref{eq:davinduction} by induction on $|\supp_B(t)|$. If $|\supp_B(t)|=1$ then $t=\alpha_1b_1$ for some $\alpha_1\neq 0$, and thus $$ \LE_\Lambda(K+t) = \sum_{x\in K|b_1^\perp \cap \Lambda|b_1^\perp} \LE_\Lambda((K+t) \cap (x+\R b_1)).$$ Since the bodies on the right hand side are segments parallel to $t$, we can apply 
\eqref{eq:davinduction1.5} and obtain \begin{equation}
\label{eq:davinduction2}
\begin{split}
\LE_\Lambda(K+t) &\leq \sum_{x\in K|b_1^\perp\cap \Lambda|b_1^\perp} \big(\LE_\Lambda (K \cap (x+\R b_1)) + 1\big)\\
 &= \LE_\Lambda (K) + \LE_{\Lambda | b_1^\perp} (K|b_1^\perp),
\end{split}
\end{equation} 
which corresponds to \eqref{eq:davinduction} in this case.

Finally, let $t=\sum_{i=1}^n \alpha_i b_i$ be an arbitrary non-zero vector in $\mathrm{span}(\Lambda)$. Consider any $j\in\mathrm{supp}_B(t)$. We define $t^\prime = t - \alpha_j b_j$ and $t^{\prime\prime} =t^\prime | b_j^\perp$ as well as $B^\prime = B\setminus\{b_j\}$ and $B^{\prime\prime} = B^\prime | b_j^\perp$. Then, we observe that $$\mathrm{supp}_{B}(t^\prime) = \mathrm{supp}_{B^{\prime\prime}}(t^{\prime\prime}) = \mathrm{supp}_B(t) \setminus \{j\}.$$ Therefore, we obtain with $\widetilde{L_I} = \spann\{ b_i|b_j^\perp: i\in I\}$ that
\begin{align*} \LE_\Lambda(K+t) = &
\LE_\Lambda(K+t^\prime +\alpha_jb_j) \\ 
 \leq & \LE_\Lambda(K+t^\prime)+ \LE_{\Lambda | b_j^\perp}((K+t^\prime)|b_j^\perp)\\
= &\LE_\Lambda(K+t^\prime)+ \LE_{\Lambda | b_j^\perp}(K|b_j^\perp + t^{\prime\prime})\\ 
\leq & \sum_{I\subseteq\mathrm{supp}_B(t)\setminus j} \LE_{\Lambda|L_I^\perp}(K|L_I^\perp)\\
 &+  \sum_{I\subseteq\mathrm{supp}_B(t)\setminus j} \LE_{(\Lambda | b_j^\perp) | \widetilde{L_I}^\perp} ( (K|b_j^\perp) | \widetilde{L_I}^\perp)\\
 =  &\sum_{I\subseteq\mathrm{supp}(t)\setminus j}\Big( \LE_{\Lambda|L_I^\perp}(K|L_I^\perp) + \LE_{\Lambda | L_{I\cup j}^\perp}(K|L_{I\cup j}^\perp)\Big)\\ 
 = &\sum_{I\subseteq\mathrm{supp}_B(t)} \LE_{\Lambda|L_I^\perp}(K|L_I^\perp).
\end{align*}
For the first inequality we used \eqref{eq:davinduction2}, and for the second inequality we used the induction hypothesis \eqref{eq:davinduction} applied to $K$, $\Lambda$, $B$ and $t^\prime$, as well as to $K|b_j^\perp$, $\Lambda|b_j^\perp$, $B^{\prime\prime}$ and $t^{\prime\prime}$. This finishes the proof of Claim \ref{claim:3} and so we obtain
\begin{equation}
\label{eq:translation}
\LE(\inter K+t) \leq \sum_{I\subseteq [n]} \LE_{\Z^n|L_I^\perp}(\inter K | L_I^\perp)
\end{equation}
for any $t\in\R^n$. Now inequality \eqref{eq:revdavi} follows from \eqref{eq:vdc}.

To see that it is tight let $K=[0,k_i]\times ...\times [0,k_n],$ where $k_i\in\Z_{>0}$, and $b_i=e_i$. Then we have
\begin{equation}
\begin{split}
\vol(K)& = \prod_{i=1}^n k_i = \prod_{i=1}^n \big((k_i-1) +1\big)\\
& = \sum_{I\subseteq [n]} \prod_{i\in I} (k_i-1) =  \sum_{ I \subseteq [n]} \LE_{\Z^n | L_I^\perp} ( \mathrm{int}K|L_I^\perp ).\qedhere
\end{split}
\end{equation}
\end{proof}

\section{Two-dimensional case}
\label{sec:dim2}

We start with the proof of Proposition \ref{prop:polygons}. By Pick's Theorem \cite[Theorem 2.8]{beck}, we have for a lattice polygon $P$ that
\begin{equation}
\label{eq:pickinter}
\LE(\inter P) =\vol(P)-\frac{\LE(\bd P)}{2}+1.
\end{equation}
An inequality of Henk, Sch\"urmann and Wills yields that \cite[Equation 1.6]{hsw}
\begin{equation}
\label{eq:hsw}
\frac{\LE(\bd P)}{2}\leq \vol(P)\Big(\frac{\sukz_1(P)}{2}+\frac{\sukz_2(P)}{2}\Big)
\end{equation}
holds, if $-P=P$.

\begin{proof}[Proof of Proposition \ref{prop:polygons}]
Let $\sukz_i = \sukz_i(P)$. Combining \eqref{eq:pickinter} with \eqref{eq:hsw} yields
\begin{equation}
\LE(\inter P) \geq \vol(P) - \vol(P)\Big(\frac{\sukz_1}{2}+\frac{\sukz_2}{2}\Big) + 1.
\end{equation}
By the upper bound in \eqref{eq:minkowski2ndtheorem}, we have $1\geq \vol(P)\sukz_1\sukz_2/4$. Hence,
\begin{equation}
\LE(\inter P)\geq \vol(P)\Big(1-\frac{\sukz_1}{2}-\frac{\sukz_2}{2}+\frac{\sukz_1\sukz_2}{4}\Big) \geq \vol(P)\Big(1-\frac{\sukz_1}{2}\Big)\Big(1-\frac{\sukz_2}{2}\Big),
\end{equation}
and the proof is finished.
\end{proof}



For the proof of Theorem \ref{thm:dim2}, we take the reduction from Section \ref{sec:anti-blocking} a step further, by shaking $K$ in such a way that it is anti blocking and, in addition, located below the diagonal line passing through $(2/\sukz_1) e_1$ and $(2/\sukz_1)e_2$ (cf.\ Figure \ref{fig:process}).

To this end, we consider non-orthogonal shakings as a generalization of the Blaschke shakings in Section \ref{sec:anti-blocking}; For an affine line $\ell\subseteq\R^2$ and a vector $u\in\R^2\setminus\{0\}$ which is not parallel to $\ell$, let $\pi_{u,\ell}$ denote the projection on $\ell$ along $u$. For $K\in\K^2$, we then define 
\begin{equation}
    \blaschke_{u,\ell}(K) = \bigcup_{x\in\pi_{u,\ell}(K)}\left[x,\, x+\vol_1\left(K\cap(x+\R u)\right)\frac{u}{|u|} \right] 
\end{equation}
as the Blaschke shaking of $K$ with respect to $u$ and $\ell$. Note that in the setting of Section \ref{sec:anti-blocking}, we have $\blaschke_u = \blaschke_{u,u^\perp}$.

As we saw in Section \ref{sec:anti-blocking}, it is enough to prove Theorem \ref{thm:dim2} for $K\in\K^2$ anti-blocking. Starting with anti-blocking body $K$ that satisfies  $\norm{e_i}{K}=\sukz_i(K)$ (cf.\ Lemma \ref{lemma:succminantiblocking}), we construct a new body $A$ by shaking $K$ first vertically and then horizontally from below against a lattice diagonal $D=\{x\in\R^2: x_1+x_2=m\}$, $m\in\Z$, and finally back down on $e_2^\perp$. (The value $m\in\Z$ may be chosen arbitrarily since lattice translations do not change the involved parameters.) Formally, we define $A=T(K)$, where 
\begin{equation}
    T = \blaschke_{e_2}\circ \blaschke_{-e_1,D}\circ \blaschke_{-e_2,D}.
\end{equation}

\begin{figure}
     \centering
     \begin{subfigure}[t]{0.5\textwidth}
         \centering
         \includegraphics[width=\textwidth]{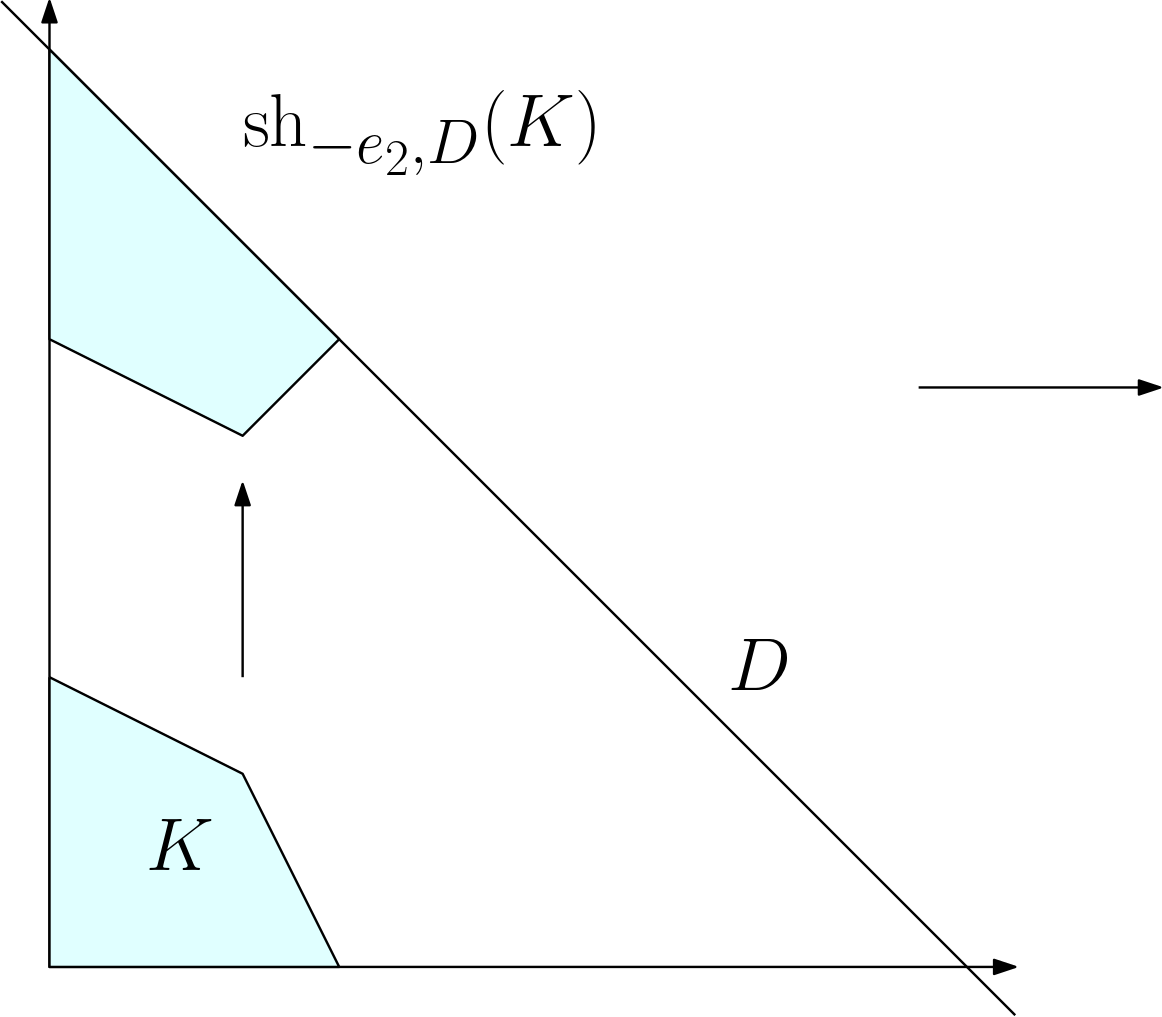}
         \caption{First, $K$ is pushed under the diagonal $D$ from below, ... }
         \label{fig:process1}
     \end{subfigure}
     \hfill
     \begin{subfigure}[t]{0.45\textwidth}
         \centering
         \includegraphics[width=\textwidth]{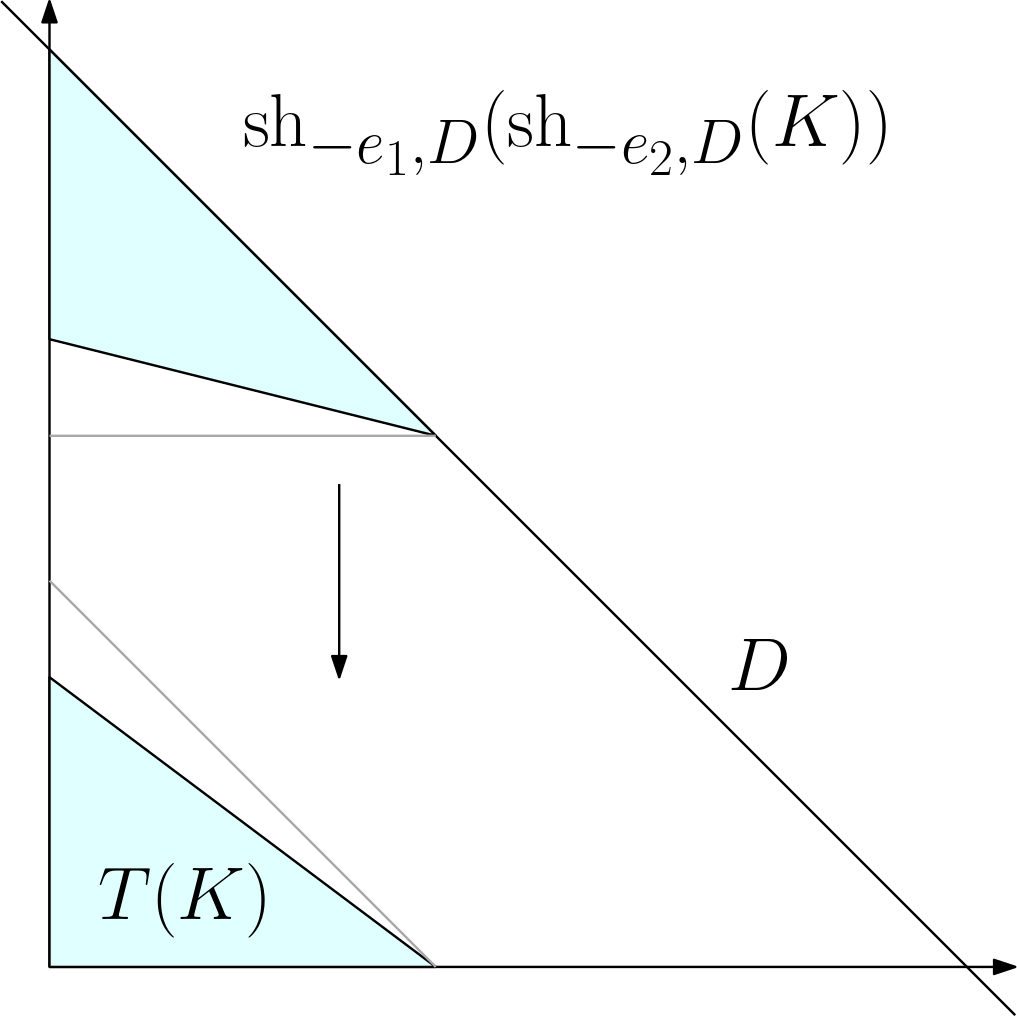}
         \caption{...next from the right and then back down on $e_2^\perp$. }
         \label{fig:process2}
     \end{subfigure}
     \hfill
        \caption{Illustration of the shaking process $T$.}
        \label{fig:process}
\end{figure}

We claim that $A$ satisfies the following properties:

\begin{lemma}
\label{lemma:A}
Let $K$ and $A$ be as above. Then the following statements hold true:
\begin{enumerate}
    \item $A$ is convex,
    \item $A$ is anti-blocking,
    \item $\vol(A) = \vol(K)$,
    \item $\LE(A) \geq \LE(K)$,
    \item $\LE(\inter A)\leq \LE(\inter K)$,
    \item $\sukz_1(A) \leq \sukz_1(K)$,
    \item $\sukz_2(A) = \sukz_2(K)$ and
    \item $A\subseteq \{x\in\R^2: x_1+x_2\leq 2/\sukz_1(A)\}$.
\end{enumerate}
\end{lemma}

We will prove Lemma \ref{lemma:A} at the end of this section. For the proof of Theorem \ref{thm:dim2}, we also need the following estimates, which follow from elementary properties of concave functions:

\begin{lemma}
\label{lemma:concavefunc}
Let $f:[a,b]\to\R$ be a concave function, then we have
\begin{equation}
\frac{1}{2}\big(f(a)+f(b)\big)(b-a) \leq \int_a^b f(t)\,\mathrm d t.
\end{equation}
Moreover, if $f'(a)$ exists, we also have
\begin{equation}
 \int_a^b f(t)\,\mathrm d t   \leq (b-a)\Big(f(a) + (b-a)\frac{1}{2}f^\prime(a)\Big).
\end{equation}
\end{lemma}

\begin{proof}
For the upper bound, let $g$ be the affine linear function given by $g(a)=f(a)$ and $g(b)=f(b)$, i.e., $g(t) = \frac{f(b)-f(a)}{b-a}(t-a) + f(a)$. By concavity, we have $f\geq g$ and therefore 
\begin{equation}
    \int_a^b f(t)\,\mathrm d t \geq \int_a^b g(t)\,\mathrm d t = \frac{1}{2}\big(f(a)+f(b)\big)(b-a).
\end{equation}
For the upper bound let $h$ be the tangent of $f$ at $a$, i.e., $h(t) = f^\prime (a)(t-a)+f(a)$. Again by concavity, we have $h\geq f$ and, thus,
\begin{equation}
    \int_a^b f(t)\,\mathrm d t \leq \int_a^b h(t)\,\mathrm d t \leq  (b-a)\Big(f(a) + (b-a)\frac{1}{2}f^\prime(a)\Big)\qedhere
\end{equation}
\end{proof}

\begin{proof}[Proof of Theorem \ref{thm:dim2}]
We write $\sukz_i = \sukz_i(K)$, $i=1,2$. In view of Theorem \ref{prop:anti-blocking} and Lemma \ref{lemma:A}, we can assume that $K$ is an anti-blocking body with $\norm{e_i}{K}=\sukz_i$ and $K\subseteq\{x\in\R^2:x_1+x_2\leq 2/\sukz_1\}$. We let $\ell_t=\{x\in\R^2:x_2=t\}$ denote the horizontal line at height $t\in\R$ and we consider
\begin{equation}
    f:\Big[0,\,\frac{2}{\sukz_2}\Big]\to\R,\, t\mapsto \vol_1(K\cap\ell_t).
\end{equation}
We observe that $\suk(K,e_2)=2/\sukz_2$ holds. Since $K$ is convex, this implies that $f$ is concave. Moreover, since $K$ is anti-blocking, $f$ is decreasing. From the inclusion $K\subseteq\{x\in\R^2:x_1+x_2\leq 2/\sukz_1\}$ it follows that 
\begin{equation}
    \label{eq:diagest}
    f(t)\leq f(0) -t = \frac{2}{\sukz_1}-t
\end{equation}
holds for all $t\in[0,2/\sukz_2]$.
Exploiting the fact that the sections $K\cap\ell_t$ are 1-dimensional, we obtain that

\begin{align}
    \LE(K) &= \sum_{i=0}^{\lfloor 2/\sukz_2\rfloor} \LE(K\cap\ell_t) \leq \sum_{i=0}^{\lfloor 2/\sukz_2\rfloor} \big(f(i) +1\big) \\
    &= \frac{1}{2}(f(0) + f(\lfloor 2/\sukz_2\rfloor) + \sum_{i=1}^{\lfloor 2/\sukz_2\rfloor}\Big(\frac{1}{2}\big(f(i-1)+f(i)\big)\Big) + \lfloor 2/\sukz_2\rfloor + 1.
\end{align}

We apply \eqref{eq:diagest} to $f(\lfloor 2/\sukz_2\rfloor)$ and the lower bound in Lemma \ref{lemma:concavefunc} to the summands $\frac{1}{2}\big(f(i-1)+f(i)\big)$ and deduce
\begin{align}
    \LE(K) &\leq \frac{2}{\sukz_1} - \frac{1}{2}\lfloor 2/\sukz_2\rfloor + \int_0^{\lfloor 2/\sukz_2\rfloor} f(t)\,\mathrm{d} t + \lfloor 2/\sukz_2\rfloor + 1\\
    &\leq \vol(K) + \frac{2}{\sukz_1}+ \frac{1}{\sukz_2} + 1\\
    & \leq \vol(K)\Big( 1+\sukz_2 + \frac{\sukz_1}{2} +\frac{\sukz_1\sukz_2}{2}\Big)\\
    & = \vol(K) \Big(1+\frac{\sukz_1}{2}\Big)\big(1+\sukz_2\big),
\end{align}
where we used the lower bound in Minkowski's second theorem \eqref{eq:minkowski2ndtheorem}, applied to the summands $2/\sukz_1$, $1/\sukz_2$ and $1$ each,  to obtain the third line. This shows the upper bound \eqref{eq:dim2upper} of the theorem.

For the lower bound, we assume that $f$ is differentiable. Else, we might approximate $f$ with a linear spline $\varphi$ from below. $\varphi$ in turn can be approximated by a smooth concave function $g$ from below by rounding its corners. This function satisfies $g(0)=f(0)$, and thus, the anti-blocking convex body $$K^\prime = \{ (x,y)\in\R^2: 0\leq y \leq 2/\sukz_2,\,0\leq x\leq g(y)\}\subseteq K$$ is located underneath the diagonal $\{ x\in\R^2 :x_1+x_2 = 2/\sukz_2(K^\prime)\}$ as well.

We observe that $\lceil 2/\sukz_2 -1\rceil$ is the height of the highest horizontal integer line that intersects $\inter K$. Therefore, we can estimate the number of interior lattice points of $K$ as follows:
\begin{align}
    \LE(\inter K) &= \sum_{i=1}^{\lceil 2/\sukz_2 -1\rceil} \LE(\inter K\cap\ell_i) \geq \sum_{i=1}^{\lceil 2/\sukz_2 -1\rceil}\big(f(i) -1\big)\\
    &= \sum_{i=1}^{\lceil 2/\sukz_2 -1\rceil}f(i) - \lceil 2/\sukz_2 -1\rceil.
\end{align}
We use the upper bound in Lemma \ref{lemma:concavefunc} in order to estimate \begin{equation}
    f(i)\geq \int_{i}^{i+1}f(t)\,\mathrm d t - \frac 12 f^\prime(i),
\end{equation}
for all $1\leq i < \lceil 2/\sukz_2 -1\rceil$ and, since $2/\sukz_2-\lceil 2/\sukz_2 -1\rceil\leq 1$, 
\begin{equation}
    f\big(\lceil 2/\sukz_2 -1\rceil\big)\geq \int_{\lceil 2/\sukz_2 -1\rceil}^{2/\sukz_2}f(t)\,\mathrm d t - (2/\sukz_2 - \lceil 2/\sukz_2 -1\rceil) \frac 12 f^\prime\big(\lceil 2/\sukz_2 -1\rceil\big).
\end{equation}
Combining this, we obtain 
\begin{align}\label{eq:someest}
    \LE(\inter K) \geq & \sum_{i=1}^{\lceil 2/\sukz_2 -1\rceil-1}\int_{i}^{i+1} f(t)\,\mathrm d t + \int_{\lceil 2/\sukz_2 -1\rceil}^{\sukz_2/2}f(t)\,\mathrm d t\\
    & + \frac{1}{2} \bigg( \sum_{i=1}^{\lceil 2/\sukz_2 -1\rceil-1} \big(-f^\prime(i)\big) - (2/\sukz_2 -\lceil 2/\sukz_2 -1\rceil) f^\prime\big(\lceil 2/\sukz_2 -1\rceil\big)\Big)\\
    & - \lceil 2/\sukz_2 -1\rceil\\
    =& \,\vol(K) - \int_0^1 f(t)\,\mathrm d t\\
    & + \frac{1}{2} \bigg( \sum_{i=1}^{\lceil 2/\sukz_2 -1\rceil-1} \big(-f^\prime(i)\big) - (2/\sukz_2 -\lceil 2/\sukz_2 -1\rceil) f^\prime\big(\lceil 2/\sukz_2 -1\rceil\big)\Big)\\
    & - \lceil 2/\sukz_2 -1\rceil\\
\end{align}
Due to \eqref{eq:diagest}, we have 
\begin{equation}
\label{eq:int01est}
    \int_0^1 f(t)\,\mathrm d t \leq \frac{2}{\sukz_1}-\frac{1}{2}.
\end{equation}
Next, we estimate the bracket term in the last but one line of \eqref{eq:someest}. To this end, we observe that $-f^\prime$ is increasing, since $f$ is concave, and that $-f^\prime$ is non-negative, since $K$ is anti-blocking. Therefore, we obtain that
\begin{align}
    & \sum_{i=1}^{\lceil 2/\sukz_2 -1\rceil-1} \big(-f^\prime(i)\big) - (2/\sukz_2 -\lceil 2/\sukz_2 -1\rceil) f^\prime\big(\lceil 2/\sukz_2 -1\rceil\big)\\
    \geq & \sum_{i=1}^{\lceil 2/\sukz_2 -1\rceil-1} \int_{i-1}^i -f^\prime(t)\,\mathrm d t\\
    & + \big( (2/\sukz_2 -1) - (\lceil 2/\sukz_2 -1\rceil-1)\big) (-f^\prime(2/\sukz_2-1))\\
    \geq & \int_{0}^{2/\sukz_2-1}-f^\prime(t)\,\mathrm d t = f(0)-f(2/\sukz_2-1)\\
    \geq & \frac{2}{\sukz_2}-1,
\end{align}
where we used \eqref{eq:diagest} in the last step. Substituting this and \eqref{eq:int01est} into \eqref{eq:someest} yields 
\begin{align}
    \LE(\inter K) & \geq \vol(K) - \frac{2}{\sukz_1} + \frac{1}{2} + \frac{1}{2}\Big(\frac{2}{\sukz_2}-1\Big) - \left\lceil \frac{2}{\sukz_2} -1\right\rceil\\
    &\geq \vol(K)-\frac{2}{\sukz_1}-\frac{1}{\sukz_2} = \vol(K)\Big(1-\frac{1}{\vol(K)}\Big(\frac{2}{\sukz_1}+\frac{1}{\sukz_2}\Big)\Big)\\
    &\geq \vol(K)\Big(1-\frac{\sukz_1}{2}-\sukz_2\Big),
\end{align}
where we used $\vol(K) \geq \frac{2}{\sukz_1\sukz_2}$ in the last step. Therefore, the proof of \eqref{eq:dim2lower} is finished. 

\end{proof}

It remains to prove Lemma \ref{lemma:A}. We will use that non-orthogonal Blaschke shakings are monotonous:

\begin{lemma}
\label{lemma:inclusion}
Let $K,L\in\K^2$ with $K\subseteq L$. Also, let $\ell\subseteq\R^2$ be a line and $u\in\R^n\setminus\{0\}$ be not parallel to $\ell$. Then we have
$$\blaschke_{u,\ell}(K) \subseteq\blaschke_{u,\ell}(L).$$
\end{lemma}

This is a widely known fact in the context of classical Blaschke shakings (cf.\ \cite[Lemma 1.1 (iii)]{ccg}).

\begin{proof}
Let $\blaschke=\blaschke_{u,\ell}$ and $\pi = \pi_{u,\ell}$. Consider a point $x\in\blaschke(K)$. Then we have $x\in\pi(K)\subseteq\pi(L)$. Also, by inclusion, we have  $$l_1:=\vol_1\big(K\cap (x+\R u)\big)\leq\vol_1\big(L\cap (x+\R u)\big)=:l_2.$$
 Hence, since $x\in\blaschke(K)$, \[ x\in \Big[\pi(x), \pi(x) + l_1\frac{u}{|u|}\Big]\subseteq  \Big[\pi(x), \pi(x) + l_2\frac{u}{|u|}\Big] \subseteq \blaschke(L).\hfill \qedhere \]
\end{proof} 

\begin{proof}[Proof of Lemma \ref{lemma:A}]

For i) we show that $\blaschke_{u,\ell}(K)$ is convex for all $u$, $K$ and $\ell$ as in Lemma \ref{lemma:inclusion}. To see this, we consider  $x,y\in \blaschke_{u,\ell}(K)$. Let $\overline{x}$ and $\overline{y}$ denote the points in $K$ on the lines $x+\R u$ and $y+\R u$ respectively that minimize $\langle \cdot, u\rangle$. Then, the points $$\widetilde{z}=\overline{z} + |z-\pi_{u,\ell}(z)|\cdot u/|u|,\quad z\in\{x,y\},$$ are contained in $K$. Lemma \ref{lemma:inclusion} then yields that     $$ \conv\{\pi_{u,\ell}(x), \pi_{u,\ell}(y), x,y\} = \blaschke_{u,\ell}\big(\conv\{\overline{x},\widetilde{x}, \overline{y}, \widetilde{y}\}\big) \subseteq\blaschke_{u,\ell}(K). $$
In particular, $[x,y]\subseteq\blaschke_{u,\ell}(K)$, which shows that $\blaschke_{u,\ell}(K)$ is convex. Thus, $A$ is convex as well.
    
Next, we consider the box $B = [0,2/\sukz_1(K)]\times[0,2/\sukz_2(K)]$. Clearly, we have  $K\subseteq B$ and by Lemma \ref{lemma:inclusion} it follows that $A\subseteq T(B)$. The vertical segments $\blaschke_{-e_2,D}$ are of length $2/\sukz_2(K)$. The vertical segments in $\blaschke_{-e_1,D}(\blaschke_{-e_2,D}(B))$ are also not longer than those in $\blaschke_{-e_2,D}(K)$. Therefore, all vertical segments in $T(B)$ (and thus also in $A$) are of length at most $2/\sukz_2(K)$ (cf.\ Figure \ref{fig:boxshaking}). On the other hand, by considering the triangle 
$$\Delta=\conv\{0,2/\sukz_1(K) e_1, 2/\sukz_2(K) e_2\}\subseteq K,$$ which fulfills $T(\Delta)= \Delta$ because of $\sukz_1(K)\leq\sukz_2(K)$, we see that the segment over the origin in $A$ has length precisely $2/\sukz_2(K)$. Since by construction we have $A\cap e_2^\perp = A|e_2^\perp$, we obtain from this  $A\cap e_1^\perp= A|e_1^\perp $ as well. Therefore, $A$ is anti-blocking and fulfills $\norm{e_2}{A}=\sukz_2(K)$, as well as $\norm{e_1}{A}\leq\norm{e_1}{\Delta} = \sukz_1$. So we obtained ii), vi) and vii). 

\begin{figure}
     \centering
     \begin{subfigure}[t]{0.5\textwidth}
         \centering
         \includegraphics[width=\textwidth]{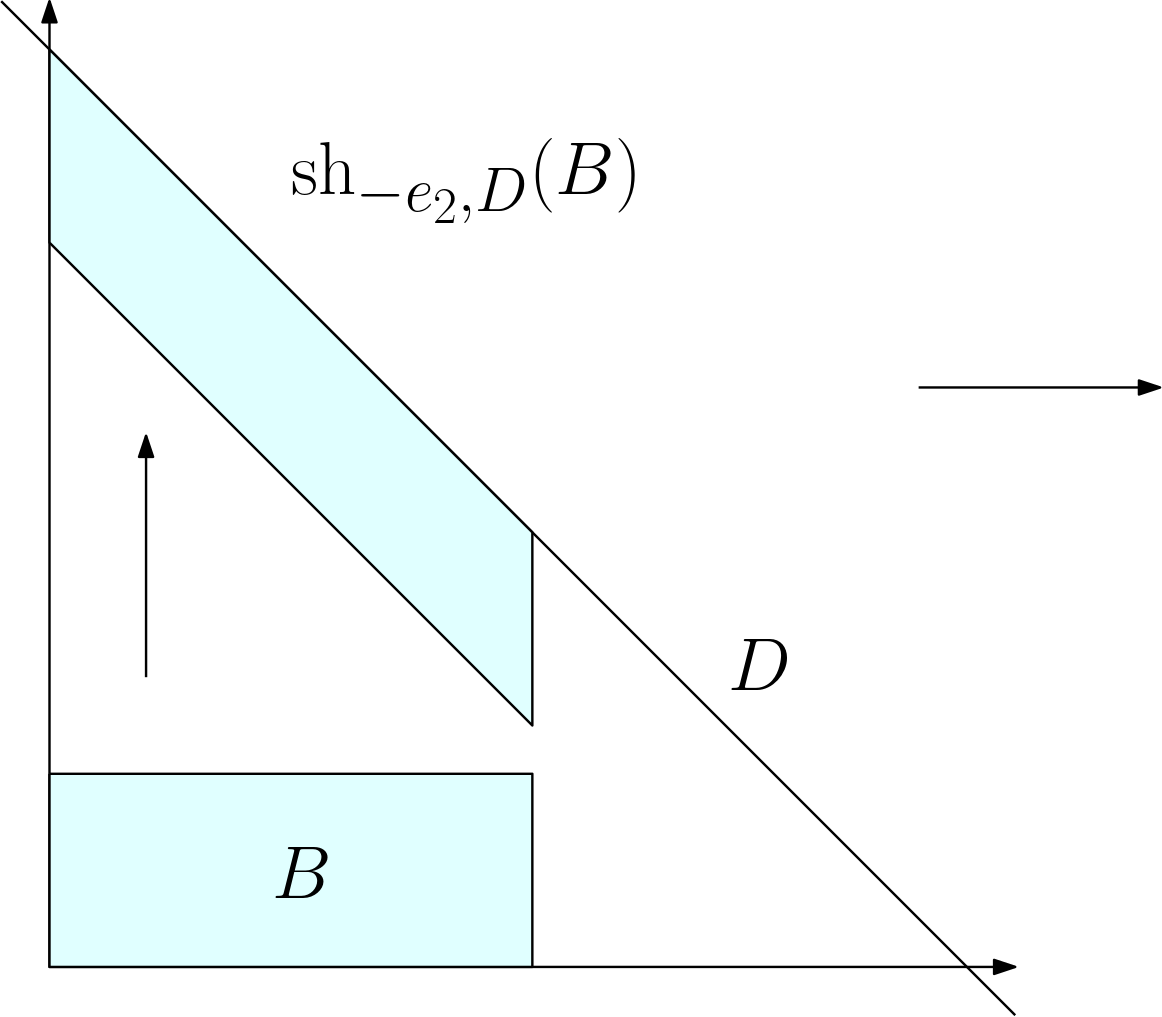}
         \caption{Applying $\blaschke_{-e_1,D}$ to the box $B$ does not change the lengths of the vertical segments.}
         \label{fig:sh-e2}
     \end{subfigure}
     \hfill
     \begin{subfigure}[t]{0.45\textwidth}
         \centering
         \includegraphics[width=\textwidth]{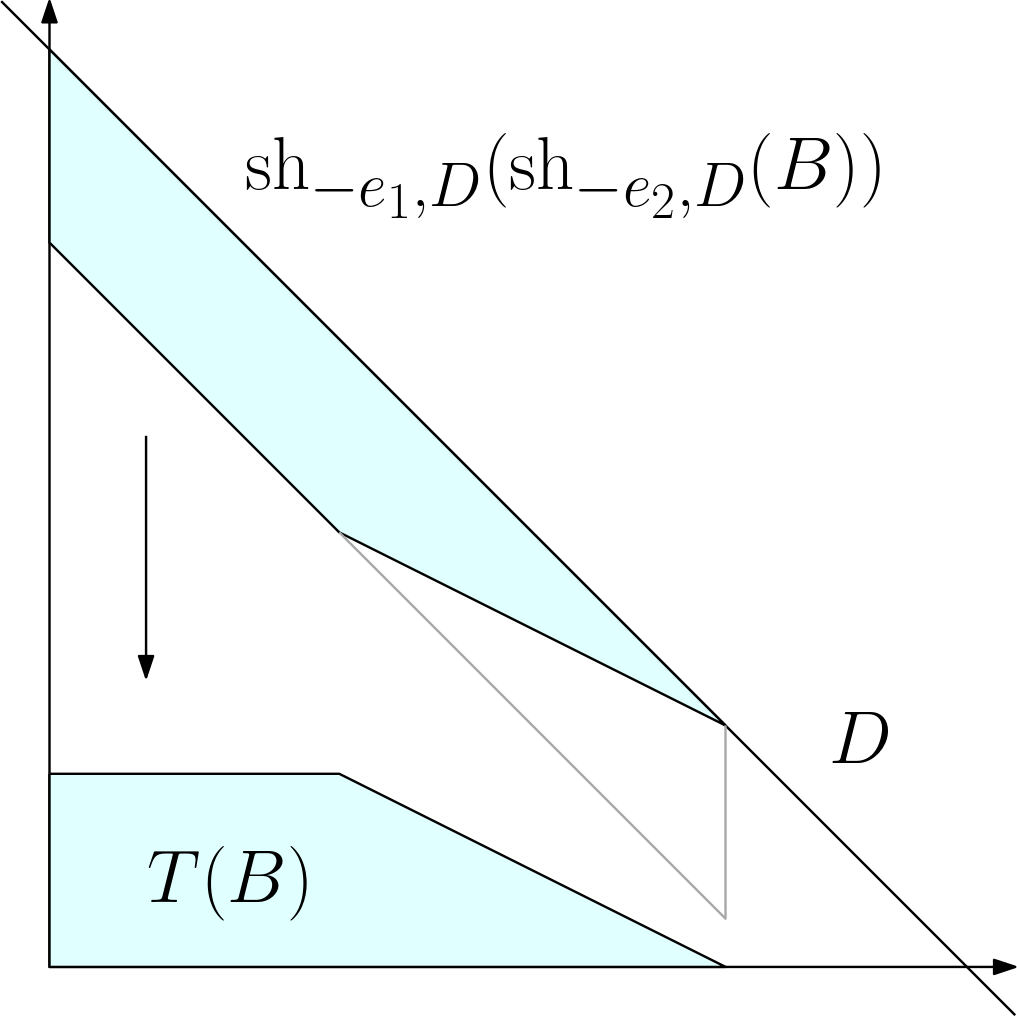}
         \caption{Applying $\blaschke_{-e_1,D}$ does not lead to any vertical segments of length higher than $ 2/\sukz_2(K)$ as they are all contained in the grey area.}
         \label{fig:she2}
     \end{subfigure}
     \hfill
        \caption{Behaviour of the vertical segments when passing from $B$ to $T(B)$}
        \label{fig:boxshaking}
\end{figure}

iii) follows from Fubini's Theorem applied to $\spann\{u\}$ and $u^\perp$, since also for arbitrary non-orthogonal shakings one has that $\blaschke_{u,\ell}(K)|u^\perp=K|u^\perp$ and $\vol_1(\blaschke_{u,\ell}(K) \cap (x+\R u))=\vol_1(K\cap (x + \R u))$, for any $x\in u^\perp$.

iv) and v) are proven in the same way as Lemma \ref{lemma:blaschke} v) and vi), since $\pi_{-e_i,D}(\Z^2) = \Z^2\cap D$ and, thus, all the shaken lattice segments in $e_i$-direction of any $\blaschke_{-e_i,D}(K)$ contain a lattice point on $D$ as an endpoint.

For viii), let $p$ denote the lowest point on the diagonal (with respect to $e_2$) such that $p\in \blaschke_{-e_1,D}(\blaschke_{-e_2,D}(K))$. Then we have  
\begin{equation}
    \blaschke_{-e_1,D}(\blaschke_{-e_2,D}(K))\subseteq \conv\{me_2, p|e_1^\perp, p\}, 
\end{equation}
where $m$ is the integer with the property that $me_2\in D$ (cf.\ Figure \ref{fig:process}). Applying $\blaschke_{e_2}$ to both sides of the inclusion yields that
\[A\subseteq\conv\{0,2/\sukz_1(A) e_1, 2/\sukz_1(A)e_2\}\subseteq\{x\in\R^2:x_1+x_2\leq 2/\sukz_1(A)\}.\qedhere\]
\end{proof}

\section{Acknowledgments}

The authors wish to thank Martin Henk and Matthias Schymura for helpful comments on the paper.


\vspace*{-0.5cm} 

\end{document}